\documentclass[12pt,reqno]{amsart}

\usepackage{amsthm}
\usepackage{amsmath}
\usepackage{amssymb}
\usepackage{latexsym}
\usepackage{enumitem}
\usepackage{hyperref}
\usepackage{url}
\usepackage[left=3.5cm, right=3.5cm]{geometry}

\hypersetup{
    pdftitle={Small doubling in ordered semigroups},
    pdfauthor={Salvatore Tringali},
    pdfmenubar=false,
    pdffitwindow=true,
    pdfstartview=FitH,
    colorlinks=true,
    linkcolor=blue,
    citecolor=blue,
    urlcolor=cyan
}

\newtheorem{thm}{Theorem}
\newtheorem{lem}{Lemma}
\newtheorem{prop}{Proposition}[section]
\newtheorem{cor}{Corollary}

\theoremstyle{definition}
\newtheorem{question}{Question}

\newtheorem{exa}{Example}[section]
\newtheorem{rem}{Remark}

\DeclareMathOperator{\GL}{GL}
\DeclareMathOperator{\CC}{C}
\DeclareMathOperator{\NN}{N}

\renewcommand{\AA}{\mathbb A}

\begin{document}
\title{Small doubling in ordered semigroups}

\author{Salvatore Tringali}

\address{CMLS, \'Ecole polytechnique - {}91128 Palaiseau cedex, France
\newline
\indent{}\textit{Web site:} {\rm \href{http://www.math.polytechnique.fr/~tringali/}{http://www.math.polytechnique.fr/\textasciitilde{}tringali/}}}

\email{salvatore.tringali@cmls.polytechnique.fr }

\thanks{This research is partially supported by the French ANR Project No. ANR-12-BS01-0011 (project CAESAR). Some fundamental aspects of the work were developed while the author was funded from the European Community's 7th Framework Programme (FP7/2007-2013) under Grant Agreement No. 276487 (project ApProCEM)}




\begin{abstract}
Let $\mathbb{A} = (A, \cdot)$ be a semigroup. We generalize some recent results by G. Freiman, M. Herzog and coauthors on the
structure theory of set addition from the context of linearly orderable groups to linearly orderable semigroups, where we say that $\mathbb{A}$ is linearly orderable if there exists a total order $\le$ on $A$ such that $xz < yz$ and $zx < zy$ for all $x,y,z \in A$ with $x < y$.

In particular, we
find that if $S$ is a finite subset of $A$ generating a non-abelian
subsemigroup of $\mathbb{A}$, then $|S^2| \ge 3|S|-2$. On the road to this goal, we also prove a
number of subsidiary results, and most notably that for $S$ a finite subset of $A$ the
commutator and the normalizer of $S$ are equal to each other.
\end{abstract}

\maketitle

\section{Introduction}\label{sec:intro}
Semigroups are ubiquitous in mathematics. Apart from
being a subject of continuous
interest to algebraists, they provide a natural framework for introducing
several broadly-scoped concepts and developing large parts of
theories traditionally presented in much less general contexts.
While on the one hand this
makes it
possible to use methods and results otherwise restricted to ``richer
settings'' for larger
classes of problems, on the other hand it can suggest new directions of research and shed light on classical questions, say, with a primary focus on groups.

Through the present paper, a semigroup is, as usual, a pair $\AA = (A, \cdot)$ consisting
of a set $A$, called the carrier of $\AA$, and an associative binary operation $\cdot$ on $A$
(unless otherwise specified, all semigroups considered below are written
multiplicatively). Then, for $S \subseteq A$ we write
$\langle S \rangle_\AA$ for the smallest subsemigroup of $\mathbb A$ containing $S$, which is simply denoted by $\langle S \rangle$ if $\mathbb A$ is implied from the context.

We let an \textit{ordered semigroup} be a triple $(A, \cdot{}, \le)$, where $(A, \cdot{})$ is a semigroup, $\le$ is an order on $A$ (notice that, in this work, the term ``order'' always means ``total order''; see also Section \ref{sec:generalities}), and the following holds:
\begin{equation}\label{equ:totally_ordered_sgrp_a011}
\forall a, b, c \in A: a < b \quad \Longrightarrow \quad a c \le b c\ \ {\rm and }\ \ c a \le c b.
\end{equation}
If each of the signs ``$\le$'' in \eqref{equ:totally_ordered_sgrp_a011} is replaced with the sign ``$<$'', then $(A, \cdot{}, \le)$ is called a \textit{linearly ordered semigroup}; see, e.g., \cite{Iwasa48}.

Accordingly, we say that a semigroup $\AA = (A, \cdot{})$ is \textit{[linearly] orderable} if there exists an order $\le$ on $A$ such that $(A, \cdot{}, \le)$ is a [linearly] ordered semigroup. Then, we may also say that $\AA$ is [linearly] ordered by $\le$.

All of the above notions and terminology are adapted to monoids (that is, unital semigroups) and groups in the obvious way.

Our interest in semigroups is related here to the structure theory of groups
and its generalizations; this is an active area of
research, which has drawn a constantly increasing attention in the last two decades,
and has led to significant progress in several fields, from
algebra \cite{Gero12} to number theory and
combinatorics \cite{Natha96, Ruzsa09, TaoVu09}.

The present paper fits into this context. Our primary goal is, in fact, to extend some recent results by
G. A. Freiman, M. Herzog and coauthors from the setting of linearly orderable groups \cite{Frei12} to linearly
orderable \textit{semi}groups.

Specifically, assume for the remainder of this section that $\AA = (A, \cdot)$ is a fixed semigroup (unless a statement to the contrary is made). Then, the main contribution of this work is
the following
generalization of \cite[Theorem 1.2]{Frei12} (if $S$ is a set, we use $|S|$ for its cardinality):
\begin{thm}
\label{th:main}
Let $\AA$ be a linearly orderable semigroup and $S$
a finite subset of $A$ such that $|S^2| \le 3|S|-3$.
Then $\langle S \rangle$ is abelian.
\end{thm}
This counts as a genuine generalization of \cite[Theorem 1.2]{Frei12} because, if $\AA$ is a group and $S$ is a non-empty subset of $A$ such that the smallest sub\textit{semigroup} of $\AA$ containing $S$ is abelian, then also the sub\textit{group} of $\AA$ generated by $S$ is abelian.

Our proof of Theorem \ref{th:main} basically follows the same broad scheme
as the proof of \cite[Theorem 1.2]{Frei12}, but there are significant differences
in the details. As expected, the increased generality implied by the
switching to semigroups - and especially the fact that inverses are no
longer available - presents, in practice, a number of challenges and requires
something more than a mere adjustment of terminology (in some cases, for
instance, it is not even clear \textit{how} a certain statement on
linearly ordered groups should be rephrased in the language of semigroups).

In particular, we will look for an extension of several classical results, such as the
following lemma (here and later, the lower case Latin letters $i$, $m$ and $n$ shall denote positive integers unless otherwise noted):
\begin{lem}
\label{lem:ab=ba}
Let $\AA$ be a linearly orderable semigroup
and pick $a,b \in A$. If $a^n b = ba^n$ for some
$n$, then $ab = ba$.
\end{lem}
This is, in fact, a generalization of an old lemma by N.~H.~Neumann \cite{Neum49} on commutators of linearly ordered groups, appearing as
Lemma 2.2 in \cite{Frei12}.

In the same spirit, we will also need to extend \cite[Proposition
2.4]{Frei12}. To this end, we shall use $\CC_\AA(S)$ for the centralizer of $S$ (relative to $\AA$),
viz the set of all $a \in A$
such that $a y = y a$ for every $y \in S$, and $\NN_\AA(S)$ for the normalizer of $S$ (relative to $\AA$), namely the set $
\{a \in A: aS = Sa\}$. These are written
as $\CC_\AA(a)$ and $\NN_\AA(a)$, respectively,
if $S = \{a\}$ for some $a$. Then we have:
\begin{lem}
\label{lem:lower_bounds0}
Let $\AA$ be a linearly orderable semigroup and $S$ a non-empty finite
subset of $A$, and pick $y \in A \setminus \CC_\AA(S)$. Then $
|yS \cup Sy| \ge |S| + 1$,
that is $yS \ne Sy$.
\end{lem}
Lemma \ref{lem:lower_bounds0} is proved in Section \ref{sec:basic_properties}, along with the following
 generalization of \cite[Corollary 1.5]{Frei12}, which may perhaps be interesting \textit{per se}:
\begin{thm}
\label{thm:normal_central}
Let $S$ be a finite subset of $A$ and assume that $\AA$ is a linearly orderable semigroup. Then $\NN_\AA(S) = \CC_\AA(S)$.
\end{thm}
We conclude the paper with a number of examples (Appendix
\ref{sec:examples}), mostly finalized to explore conditions under which certain
semigroups (or related structures as semirings) are linearly orderable. This is mainly
to show that the class of linearly orderable semigroups
is not, in some sense, trivial.

In particular, we prove (Theorem \ref{th:zig_zag_ordering}) that, for each
$n$, the subsemigroup of $\GL_n(\mathbb R)$, the general
linear group of degree $n$ over the real field, consisting of all upper
(respectively, lower) triangular matrices with positive entries on or above
(respectively, below) the main diagonal is linearly orderable.

Then, we raise the question (to which we do not have an answer) whether or not the same conclusion holds for the subsemigroup of $\GL_n(\mathbb R)$ consisting of those matrices
which can be written as a (finite)
product of upper or lower triangular matrices of the same type as above.

\section{General notation and definitions}
\label{sec:generalities}
We refer to \cite{BourSetTh}, \cite{BourAlgI}, and \cite{Howie96}, respectively, for notation and terminology from set theory, algebra, and semigroup theory used but not defined here.

An order on a set $A$ is a binary relation $\le$ on $A$ which is reflexive, antisymmetric, transitive, and \textit{total}, in the sense that for all $a,b \in A$ we have either $a \le b$ or $b < a$, where $<$ is used for the strict order induced on $A$ by $\le$. We write $\ge$ and $>$, respectively, for the dual order of $\le$ and $<$, as usual.

If $\AA = (A, \cdot{})$ is a semigroup and $S_1, \ldots, S_n$ are subsets of $A$, we let
$S_1  \cdots  S_n $ denote the \textit{product set}, relative to $\AA$, of the $n$-tuple $(S_1, \ldots, S_n)$, namely the set $$\{a_1
 \cdots a_n: a_1 \in S_1, \ldots, a_n \in S_n\},$$ and we write it as $S^n$ when the $S_i$ are all equal to the same $S$. In
particular, if $a \in A$, $T
\subseteq  A$ and no confusion can arise, we use $aT$ for $\{a\} T$ and $Ta$ for $T\{a\}$.

\section{Preliminaries}\label{sec:basic_properties}
In what follows, unless otherwise specified, $\AA = (A, \cdot)$ is a fixed semigroup and $\le$ is an order on $A$ for which $\mathbb A_\sharp = (A, \cdot, \le)$ is an ordered semigroup.

In this section, we collect some results that will be essential, later in Section \ref{sec:main_results}, to prove the main contributions of the paper. Some are quite elementary, and their group
analogues are part of the folklore; however, we do not have a reference to something similar for semigroups, and thus we include them
here for the sake of exposition.
In particular, the proof (by induction) of the proposition below is straightforward from the definitions, and we may omit the details.

\begin{prop}\label{prop:elementary_properties}
The following holds:
\begin{enumerate}[label={\rm(\roman{*})}]
\item If $a_1, b_1, \ldots, a_n, b_n \in A$ and $a_1 \le b_1$, \ldots, $a_n \le b_n$, then $a_1 \cdots a_n \le b_1 \cdots b_n$; also, $a_1 \cdots a_n < b_1 \cdots b_n$ if $\mathbb A_\sharp$ is linearly ordered and $a_i < b_i$ for each $i$.
\item If $a, b \in A$ and $a \le b$, then $a^n \le b^n$ for all $n$, and in fact $a^n < b^n$ if $\mathbb A_\sharp$ is linearly ordered and $a < b$.
\item If $a \in A$ is such that $a^2 \le a$, then $a^n \le a^m$ for $m \le n$; moreover, $a^n < a^m$ if $\AA_\sharp$ is linearly ordered, $a^2 < a$ and $m < n$.
\end{enumerate}
\end{prop}

Pick an element $a \in A$. We say
that $a$ is cancellable (in $\AA$) if
both of the maps $A \to A: x \mapsto ax$ and $A \to A: x \mapsto xa$ are
one-to-one. The semigroup $\AA$ is then cancellative if each element of $A$ is
cancellable.

\begin{rem}
\label{rem:cancellative_if_linearly_orderable}
A cancellative semigroup is linearly orderable if and only if it is totally
orderable. Furthermore, any linearly orderable semigroup
is cancellative.

Thus, one thing seems worth mentioning before proceeding: While, on the one hand, every commutative cancellative semigroup embeds as a subsemigroup into a group (as it follows from the standard construction of the group of fractions of a commutative monoid; see \cite[Chapter I, Section 2.4]{BourAlgI}), nothing similar is true, on the other hand, in the non-commutative case, no matter if we restrict to linearly orderable finitely generated semigroups, as first noticed by R.~E.~Johnson \cite{Johnson69} on the basis of an example by A.~Malcev \cite{Malcev37}.

This is of fundamental importance here, as it shows that the study of sumsets in linearly ordered semigroups cannot be systematically reduced, in the absence of commutativity, to the case of groups (at least, not in any obvious way).
\end{rem}
On another hand, $a \in A$ is said to be
periodic (in $\AA$) if
there exist positive integers $n$ and $p$ such that $a^n = a^{n+p}$; we then refer to
the smallest $n$ with this property as the index of $a$ (in $\AA$) and to the smallest $p$
relative to such an $n$ as the period of $a$ (in $\AA$); see, e.g., \cite[p. 10]{Howie96}. In particular, $a$ is called idempotent (in $\AA$) if it has period and index equal to $1$, namely $a = a^2$, and we say that $\AA$ is torsion-free
if its only periodic elements are idempotent.

\begin{rem}\label{rem:idempotents_behave_like_identities}
The unique idempotent element of a cancellative monoid is the identity, so
that torsion-free groups are definitely a special type of torsion-free
semigroups; cf. Example \ref{exa:torsion-freeness}. Moreover, if $\mathbb A$ is cancellative and $a \in A$ is idempotent, then $\AA$ is unital (which applies especially to linearly orderable
semigroups, in view of Remark \ref{rem:cancellative_if_linearly_orderable}):
For, $a^2 = a$ implies $a^2 b = ab$ and $ba^2 = ba$ for every $b \in A$, hence $ab = ba = b$. This ultimately
proves that $a$ serves as the identity of $\AA$.
\end{rem}

The following proposition generalizes properties mentioned in \cite[Section 2]{Frei12}.
\begin{prop}\label{prop:optional}
Let $\mathbb A_\sharp$ be a linearly ordered semigroup. We have:
\begin{enumerate}[label={\rm(\roman{*})}]
\item\label{prop:optional:item_i} If $a \in A$ and $a^2 < a$, then $ab < b$ and $aba
< b$ for all $b \in A$.
\item\label{prop:optional:item_ii} If $aba = b$ for $a,b \in A$, then $\AA$ is
unital and $a$ is the identity of $\AA$.
\item\label{prop:optional:item_iii} None of the elements of $A$ has finite period unless
$\AA$ is unital and such an element is the identity. In particular,
$\AA$ is torsion-free.
\end{enumerate}
\end{prop}

\begin{proof}
\ref{prop:optional:item_i} Pick $a,b \in A$ with $a^2 < a$. Then $a^2 b < ab$, whence
$ab < b$ by the totality of $\le$ and Remark
\ref{rem:cancellative_if_linearly_orderable}. It follows from Proposition
\ref{prop:elementary_properties} that $aba^2 < ba$; thus, $aba < b$ by
the same arguments as above.

\ref{prop:optional:item_ii} Let $a,b \in A$ be such that $aba = b$. By duality, we may suppose that $a^2 \le a$. If $a^2 < a$, then $aba < b$ by the previous point \ref{prop:optional:item_i}. Therefore, we must have $a^2 = a$,
which implies the claim by Remark \ref{rem:idempotents_behave_like_identities}.

\ref{prop:optional:item_iii} This is immediate from the above (we leave the details to the reader).
\end{proof}
The next proposition, of which we omit the proof,
is in turn an extension of an elementary property of the integers; see, for instance,
\cite[Exercise 1, p. 93]{Ruzsa09} and contrast with \cite[Theorem 1.1]{Frei12}.
\begin{prop}
\label{prop:lower_bound_on_product_with_many_factors}
Assume $\AA_\sharp$ is a
linearly ordered semigroup and let $S_1, \ldots, S_n$ be non-empty finite subsets of $A$. Then
\begin{equation}\label{equ:lower_bound_on_|S1S2Sn|_clone}
\textstyle |S_1 \cdots S_n|
\ge 1 -  n + \sum_{i=1}^n |S_i|.
\end{equation}
Moreover, \eqref{equ:lower_bound_on_|S1S2Sn|_clone} is sharp, the lower bound being
attained, e.g., by picking $a \in A$ and letting $S_i$ be, for each $i$, of the form $\{a, \ldots, a^{s_i}\}$ for some positive integer $s_i$.
\end{prop}
In particular, the second part of Proposition \ref{prop:lower_bound_on_product_with_many_factors} follows from considering that, if $\AA$ is a linearly orderable non-trivial non-empty semigroup, point \ref{prop:optional:item_iii} of Proposition
\ref{prop:optional} provides at least one
element $a \in A$
such that $a^{j_1} \ne a^{j_2}$ for all distinct integers $j_1, j_2 \ge 1$.

Now we prove the generalizations of \cite[Lemma 2.2]{Frei12} and \cite[Proposition 2.4]{Frei12} alluded to in the introduction, while noticing that, if $\AA$ is a group with identity $1$ and $a,b \in
A$, then $[a^n, b] = 1$, for some $n$, if and only if
$a^n b = ab^n$ (the square brackets denote a commutator).
\begin{prop}\label{prop:generalized_neumann}
Let $\AA_\sharp$ be a linearly ordered semigroup and pick
$a,b \in A$. If $ab < ba$ then for every $n$ we have
\begin{equation}
\label{equ:inequality_long}
a^n b < a^{n-1} b a <
\cdots < aba^{n-1} < ba^n.
\end{equation}
\end{prop}
\begin{proof}
Assume that equation \eqref{equ:inequality_long} is true
for some $n$. Then, multiplying by $a$ on the left gives
$a^{n+1} b < a^n b a < \cdots < a^2 b a^{n-1} < aba^n$, while
multiplying by $a$ on the right yields $aba^n < ba^{n+1}$. Since $ab <
ba$, the transitivity of $\le$ implies the claim by induction.
\end{proof}

The proof of Lemma \ref{lem:ab=ba} is now an immediate
consequence of Proposition \ref{prop:generalized_neumann} (by duality, if
$\AA_\sharp$ is a linearly ordered semigroup and $a,b \in A$
then we may assume $ab \le ba$ without loss of generality), so we come to Lemma \ref{lem:lower_bounds0}.

\begin{proof}[Proof of Lemma \ref{lem:lower_bounds0}]
Assume to the contrary that $yS = Sy$. Since $y \notin \CC_\AA(S)$,
we can find an element $a_1 \in S$ such that $a_1 y \ne ya_1$, which in turn implies that there exists $a_2 \in S \setminus \{a_1\}$ such
that $ya_1 = a_2y$. Then, using that $S$ is a finite set, we get a
maximum integer $k \ge 2$ and elements $a_1, \ldots, a_k \in S$ such that
\begin{enumerate}[label={\rm (\roman{*})}]
\item $ ya_i = a_{i+1} y$ for $i = 1,
\ldots, k-1$;
\item the $a_i$ are pairwise distinct for $i = 1, \ldots, k$.
\end{enumerate}
Hence, the maximality of $k$ and $yS = Sy$ imply
$ya_k = a_h y$ for some $h = 1, \ldots, k$, with the result that $y^{i+1} a_k
= a_{h+i} y^{i+1}$ for every $i = 0, \ldots, k-h$ (by induction). In particular, it holds $y^{k-h+1}
a_k = a_k y^{k-h+1}$. Therefore, $ya_k = a_k y$ (by Lemma \ref{lem:ab=ba}),
and in fact $ya_k = ya_{k-1}$ (since $a_ky = y a_{k-1}$, by construction).

So,
Remark \ref{rem:cancellative_if_linearly_orderable} yields $a_k = a_{k-1}$,
which is however absurd because $a_i \ne a_j$ for all $i,j = 1, \ldots, k$ with $i \ne
j$. The proof is thus complete.
\end{proof}

We conclude the section with the following:

\begin{proof}[Proof of Theorem \ref{thm:normal_central}]
The claim is obvious if $S$ is empty, so assume $S \ne \emptyset$. Given $y \in
\NN_\AA(S)$ we have $yS = Sy$, and Lemma
\ref{lem:lower_bounds0} implies $y \in \CC_\AA(S)$, whence we get $\NN_\AA(S) \subseteq \CC_\AA(S)$. The other inclusion is straightforward.
\end{proof}
\section{The main result}\label{sec:main_results}
Throughout, $\AA = (A, \cdot)$ denotes a fixed semigroup (unless otherwise specified). We start with a
series of three lemmas: The two first apply to cancellative semigroups in
general, while the latter is specific to linearly orderable semigroups.

\begin{lem}\label{lem:cancellative_implies_abelian}
Let $\AA$ be a cancellative semigroup and $S$ a finite
subset of $A$ such that $\langle S \rangle $ is abelian.
If $y \in A \setminus \CC_\AA(S)$, then $S^2$ is disjoint from $yS
\cup Sy$.
\end{lem}
\begin{proof}
Pick $y \in A \setminus \CC_\AA(S)$ and assume for the sake of contradiction that $S^2$ is not disjoint from $yS \cup Sy$. Without
loss of generality, there then exist $a,b,c \in S$ such that $ab = cy$. Since $\langle S
\rangle $ is abelian, this gives that $cyc = abc = cab$, whence $ab
= yc$ (using that $\AA$ is cancellative), and finally $cy = yc$.

We claim that $xy = yx$ for all $x \in S$.
For, let $x \in S$. On the one hand, we have $abx = cyx = ycx = yxc$ (as we
have just seen that $cy = yc$). On the other hand, $xab = xcy = xyc$. But $abx = xab$
(again, by the commutativity of $\langle S \rangle $). So, in the end,
$yxc = xyc$, and hence $yx = xy$ (by the cancellativity of $c$). It follows that $y
\in \CC_\AA(S)$, which is absurd.
\end{proof}
\begin{lem}\label{lem:abel_makes_everything_commute}
Let $\AA$ be a cancellative semigroup and pick elements $a,b,x,y,z
\in A$ such that $x,y,z \in \CC_\AA(b)$ and $xy = az$ (respectively,
$xy = za$). Then $ab = ba$.
\end{lem}
\begin{proof}
By duality, we just consider the case when $xy = az$. On the one hand, $xyb = azb = abz$ since $zb = bz$; on the other hand, $baz = bxy =
xyb$ since $x,y \in \CC_\AA(b)$. Hence $abz = baz$, that is $ab = ba$
(by the cancellativity of $z$).
\end{proof}
Now, assume for the remainder of the section that $\AA$ is turned into an ordered semigroup by a certain order $\le$, and set $\AA_\sharp = (A, \cdot, \le)$ for brevity.
\begin{lem}\label{lem:lower_bounds_on_a_union}
Let $\AA_\sharp$ be linearly ordered, and let $S$ be
a non-empty finite subset of $A$. Pick $y \in A
\setminus \CC_\AA(S)$. If $\langle S \rangle $ is abelian,
then $$|S^2
\cup yS \cup Sy| \ge 3|S|.$$
\end{lem}
\begin{proof}
The inclusion-exclusion principle, Remark \ref{rem:cancellative_if_linearly_orderable} and Lemma \ref{lem:cancellative_implies_abelian} give
\begin{displaymath}
|S^2 \cup yS \cup Sy| = |S^2| + |yS \cup Sy| - |S^2 \cap (yS \cup Sy)| = |S^2| +
|yS \cup Sy|,
\end{displaymath}
which is enough to complete the proof on account of the fact that
$|S^2| \ge 2|S|-1$, by Proposition \ref{prop:lower_bound_on_product_with_many_factors}, and
$|yS \cup Sy| \ge |S|+1$, by
Lemma \ref{lem:lower_bounds0}.
\end{proof}
So at long last we are ready to prove the main theorem of the paper.
\begin{proof}[Proof of Theorem \ref{th:main}]
Write $I_m$ for $\{1, \ldots, m\}$, where $m = |S|$, and let $a_1, \ldots, a_m$ be
a numbering of $S$ for which $a_1
< \cdots < a_m$. It is evident that $m \ge 2$. If $m = 2$ then $|S^2|
\le 3$, and in fact $|S^2| = 3$ by Proposition \ref{prop:lower_bound_on_product_with_many_factors}. Since
$a_1^2 < a_1 a_2 < a_2^2$ and $a_1^2 < a_2 a_1 < a_2^2$, it
follows that $S^2 = \{a_1^2, a_1 a_2, a_2^2\}$ and $a_1 a_2 = a_2 a_1$, which
implies that $\langle S \rangle $ is abelian, as desired.

So, in what follows, let $m \ge 3$ and suppose that $\langle B
\rangle $ is abelian for every subset $B$ of $A$ for which $2 \le |B| < m$ and $|B^2| \le 3|B| - 3$. Furthermore, assume by contradiction that $\langle S \rangle $ is \textit{not} abelian,
and accordingly denote by $i$ the maximum integer in $I_m$ such that $\langle T
\rangle $ is abelian for $T = \{a_1,  \ldots, a_i\}$. Then $1
\le i < m$ and $a_{i+1} \notin \CC_\AA(T)$, so on the one hand
\begin{equation}\label{equ:T^2cap_ai+1T}
T^2 \cap (a_{i+1} T \cup Ta_{i+1}) = \emptyset,
\end{equation}
thanks to Remark \ref{rem:cancellative_if_linearly_orderable} and Lemma
\ref{lem:cancellative_implies_abelian}, and on the other hand
\begin{equation}\label{equ:lower_bound_on_T^2UyTUTy}
|T^2 \cup a_{i+1} T \cup Ta_{i+1}| \ge 3i,
\end{equation}
by virtue of Lemma \ref{lem:lower_bounds_on_a_union}. Also, there exists a
positive integer $j \le i$ such that
\begin{equation}\label{equ:ai+1_doesnt_commute_with_a_j}
a_{i+1} a_j \ne a_j a_{i+1},
\end{equation}
which is chosen here to be as great as possible, in such a way that
\begin{equation}\label{equ:maximality_of_j}
x a_{i+1} = a_{i+1} x\text{ for every } x \in T \text{
with } a_j < x.
\end{equation}
We have that $a_j \notin \CC_\AA(V)$, where $V = S \setminus T =
\{a_{i+1}, \ldots, a_m\}$, and
\begin{equation}\label{equ:outsiders}
V^2 \cap (T^2 \cup a_{i+1} T \cup T a_{i+1}) = \emptyset
\end{equation}
since $a_h a_k < a_{i+1}^2 \le a_r a_s$ for all indices $h,k,r,s \in I_m$ with
$h+k \le 2i+1$ and $i+1 \le \min(r,s)$. Then the inclusion-exclusion principle,
together with \eqref{equ:lower_bound_on_T^2UyTUTy} and the standing assumptions, gives
that
\begin{displaymath}
|V^2| \le |S^2| - |T^2 \cup a_{i+1} T \cup Ta_{i+1}| \le 3m-3 - 3i = 3|V| - 3.
\end{displaymath}
Thus $2 \le |V| < m$, and $\langle V \rangle $ is abelian (by the inductive hypothesis). Then
\begin{equation}\label{equ:V^2cap_ajVcupVaj}
V^2 \cap (a_j V \cup Va_j) = \emptyset,
\end{equation}
in view of Remark \ref{rem:cancellative_if_linearly_orderable}, Lemma
\ref{lem:cancellative_implies_abelian} and the fact that $a_j \notin \CC_\AA(V)$. We claim
\begin{equation}\label{equ:T2cap_ajV}
T^2 \cap (a_j V \cup Va_j) = \emptyset.
\end{equation}
For, assume to the contrary, with no loss of generality, that $T^2 \cap a_j V
\ne \emptyset$, namely $xy = a_j z$ for some $x,y \in T$ and $z \in V$. Using that $y
< z$, this yields $a_j < x$, and similarly $a_j < y$ as $\langle T
\rangle $ is abelian (so that $xy = yx$, and hence $yx =
a_j z$). It then follows from \eqref{equ:maximality_of_j} and the commutativity of
$\langle V \rangle $ that $x,y,z \in \CC_\AA(a_{i+1})$.
Thus, we get $a_{i+1} a_j = a_j a_{i+1}$ by Lemma \ref{lem:abel_makes_everything_commute}, which however contradicts \eqref{equ:ai+1_doesnt_commute_with_a_j} and
implies \eqref{equ:T2cap_ajV}.

With that said, let $x \in T$ and $y \in V$ be such that $xa_{i+1} = a_j y$. Since
$a_{i+1} \le y$, it is clear that $a_j \le x$. Suppose for the sake
of contradiction that $a_j < x$. Then we get from
\eqref{equ:maximality_of_j} and the commutativity of $\langle V
\rangle $ that $x, a_{i+1}, y \in \CC_\AA(a_{i+1})$, with the result that $a_j a_{i+1} = a_{i+1} a_j$ (by Lemma
\ref{lem:abel_makes_everything_commute}). But this is in open contrast with
\eqref{equ:ai+1_doesnt_commute_with_a_j}, and it is enough to argue that
\begin{displaymath}
T a_{i+1} \cap a_j V = \{a_j a_{i+1}\}.
\end{displaymath}
Thus, the inclusion-exclusion principle gives that
\begin{equation}\label{equ:ultimate_lower_bound}
|Ta_{i+1} \cup a_j V| = |Ta_{i+1}| + |a_j V| - |T a_{i+1} \cap a_j V| =
m-1,
\end{equation}
which in turn implies, together with \eqref{equ:T^2cap_ai+1T},
\eqref{equ:outsiders}, \eqref{equ:V^2cap_ajVcupVaj} and \eqref{equ:T2cap_ajV},
that
\begin{displaymath}
|T^2 \cup V^2 \cup Ta_{i+1} \cup a_j V| = |T^2| + |V^2| + |Ta_{i+1} \cup a_j V|.
\end{displaymath}
It follows from Proposition \ref{prop:lower_bound_on_product_with_many_factors} and
\eqref{equ:ultimate_lower_bound} that
\begin{displaymath}
|T^2 \cup V^2 \cup Ta_{i+1} \cup a_j V| \ge (2i - 1) + (2m - 2i - 1) + (m-1) =
3m - 3.
\end{displaymath}
As $|S^2| \le 3m-3$ and $T^2 \cup V^2 \cup Ta_{i+1} \cup a_j V \subseteq S^2$,
it is then proved that
\begin{equation}\label{equ:mucchio}
S^2 = T^2 \cup V^2 \cup Ta_{i+1} \cup a_j V.
\end{equation}
So to conclude, let us define $a = a_{i+1} a_j$. By
\eqref{equ:T^2cap_ai+1T} and \eqref{equ:outsiders}, it is straightforward that $a
\notin T^2 \cup V^2$, and we want to show that $a \notin Ta_{i+1} \cup a_j V$ to
reach a contradiction. For, observe that, by
\eqref{equ:ai+1_doesnt_commute_with_a_j} and Lemma
\ref{lem:lower_bounds0}, there exist $x \in T$ and $y \in V$
such that
\begin{equation}\label{equ:a_few_more_steps}
a_{i+1} x \notin Ta_{i+1},\quad y a_j \notin a_j V.
\end{equation}
Since $a_{i+1} x, y a_j \notin T^2 \cup V^2$ by
\eqref{equ:T^2cap_ai+1T}, \eqref{equ:outsiders}, \eqref{equ:V^2cap_ajVcupVaj}
and \eqref{equ:T2cap_ajV}, it then follows from \eqref{equ:mucchio} that $a_{i+1}
x \in a_j V$ and $ y a_j \in Ta_{i+1}$, so we find $b \in V$ and $c \in T$ such that
\begin{equation}\label{equ:commuting_equations}
a_j b = a_{i+1} x,\quad y a_j = c\;\!a_{i+1}.
\end{equation}
Suppose that $a \in Ta_{i+1}$, i.e. there exists $z \in T$
for which $z a_{i+1} = a_{i+1} a_j$.

We get from \eqref{equ:ai+1_doesnt_commute_with_a_j} that $z \ne a_j$. If $a_j < z$ then \eqref{equ:maximality_of_j} yields $z \in \CC_\AA(a_{i+1})$, and Lemma \ref{lem:abel_makes_everything_commute} implies $a_{i+1} a_j
= a_j a_{i+1}$, again in
contradiction to \eqref{equ:ai+1_doesnt_commute_with_a_j}. Thus $z < a_j$.

In addition, $x \le a_j$, since otherwise $a_{i+1} x = x
 a_{i+1} \in T a_{i+1}$ in view of \eqref{equ:maximality_of_j}, in contradiction to
\eqref{equ:a_few_more_steps}. Considering that $\langle T \rangle $ is
abelian, it follows from \eqref{equ:commuting_equations} that
$$a_j b a_j =
a_{i+1} x   a_j = a_{i+1} a_j x.$$
But $a_{i+1} a_j =  z
a_{i+1}$, so at the end $a_j b a_j = z  a_{i+1} x$. Hence,
$b a_j < a_{i+1} x$ as $z < a_j$, which is absurd as
$a_{i+1} \le b$ and $x \le a_j$, viz $ a_{i+1}
x \le b a_j$. This gives $a \notin
Ta_{i+1}$.

Finally, assume that $a \in a_j V$, i.e. there exists $w \in V$ such that
$a_{i+1} a_j = a_j w$. By construction of $V$, we have $a_{i+1} \le w$, and in fact
$a_{i+1} < w$ by \eqref{equ:ai+1_doesnt_commute_with_a_j}. We want to show
that $c \le a_j$. For, suppose to the contrary that $a_j <
c$.
The commutativity of $\langle V \rangle $, together with
\eqref{equ:maximality_of_j}, then yields that $c, a_{i+1}, y \in \CC_\AA(a_{i+1})$, so $a_{i+1} a_j = a_j a_{i+1}$ by
\eqref{equ:commuting_equations} and Lemma
\ref{lem:abel_makes_everything_commute}; this contradicts
\eqref{equ:ai+1_doesnt_commute_with_a_j}, and hence $c \le a_j$. Using once
more that $\langle V \rangle $ is abelian, it is then immediate from
\eqref{equ:commuting_equations} that
$$a_{i+1} c  a_{i+1} = a_{i+1}y  a_j
= y  a_{i+1} a_j,$$
so $a_{i+1} c{ }a_{i+1} = y a_j w$ since
$a_{i+1} a_j = a_j w$. But, as argued before, $a_{i+1} < w$, whence it is
seen that $y a_j < a_{i+1} c$, which is absurd because
$a_{i+1} \le y$, by construction of $V$, and $c \le a_j$, as
proved above. Therefore, we get that $a \notin a_j V$.

Putting all together, it follows that $a \notin T^2 \cup V^2 \cup
Ta_{i+1} \cup a_j V$, which is however in contradiction to \eqref{equ:mucchio},
as $a$ is obviously an element of $S^2$. Thus, $\langle
S\rangle $ is abelian, and we are done.
\end{proof}
In some sense, Theorem \ref{th:main} is best possible. More precisely,
\cite[Section 3]{Frei12} provides the example of a subset $S$ of the carrier of a linearly ordered
group generating a non-abelian subgroup and such that $|S^2| = 3 |S| - 2$.
\begin{cor}
Assume $\AA_\sharp$ is a linearly orderable semigroup and let $S$ be a finite subset of $A$ generating a non-abelian subsemigroup of $\AA$. Then $|S^2| \ge 3|S| - 2$.
\end{cor}
\begin{proof}
It is just a trivial restatement of Theorem \ref{th:main}.
\end{proof}
We have not found so far an appropriate way to extend Proposition 3.1 in \cite{Frei12} from finite subsets of linearly ordered groups, generating abelian subgroups, to finite subsets of linearly ordered semigroups, generating abelian subsemigroups, so we raise the following:
\begin{question}
Assume that $\AA$ is a linearly orderable semigroup. Let $S$ be a finite subset of $A$, set $s = |S|$ and $t = |S^2|$ for the sake of notation, and
suppose that $t \le 3s - 4$ and $\langle S \rangle $ is abelian. Is it then possible to find $a,b \in A$ such that $ab = ba$ and $S$ is a subset of the progression $a, ab, \ldots, ab^{t-s}$?
\end{question}
\appendix
\section{Examples}\label{sec:examples}
We conclude the paper with a few examples. As mentioned in the introduction, the basic goal is
to show that [linearly] orderable semigroups and related structures are far from being ``exotic''.

We start with an orderable semigroup which is not linearly
orderable. Next, we mention some notable classes of linearly orderable groups and a linearly orderable monoid which is not a linearly orderable group (we do not know if it embeds into a linearly ordered semigroup).

\begin{exa}
Every set $A$ can be turned into a semigroup by the operation $\cdot{}: A \times
A \to A: (a,b) \to a$; see, for instance, \cite[p. 3]{Howie96}. Trivially, if $\le$
is a total order on $A$ then $(A, \cdot{}, \le)$ is a totally ordered
semigroup.
However,
$(A, \cdot{})$ is not linearly orderable for $|A| \ge 2$ (e.g., because it is not cancellative).
\end{exa}
\begin{exa}\label{exa:torsion-freeness}
An interesting variety of linearly orderable groups is provided by abelian
torsion-free groups, as first proved by F. W. Levi in \cite{Levi13}, and the result can be, in fact,
extended to abelian cancellative torsion-free semigroups with no substantial modification; see the comments following Remark \ref{rem:cancellative_if_linearly_orderable} in Section \ref{sec:basic_properties} and Corollary 3.4 in R. Gilmer's book on commutative
semigroup rings \cite{Gilmer}.

In a similar vein, K. Iwasawa \cite{Iwasa48}, Malcev \cite{Malc48} and B. H.
Neumann \cite{Neum49}
established independently  that torsion-free nilpotent
groups are linearly orderable.

Save for the semigroup analogue of Levi's result, all of the above is
already mentioned in \cite{Frei12}, where the interested reader can find further references to existing
literature on the subject. Two more examples (of linearly orderable groups) which are \textit{not} included in \cite{Frei12} are \textit{pure} braid groups \cite{RZ98} and free groups \cite{Iwasa48}.
\end{exa}
\begin{exa}\label{exa:honest_semigroups}
As for linearly orderable monoids which are not linearly
orderable groups, consider, for instance, the free monoid
\cite[Section 1.6]{Howie96} on a well-ordered alphabet $(X,
\le)$ together with the ``shortlex ordering'': Words are primarily sorted by
length, with the shortest ones first, and words of the same length are then
sorted into lexicographical order.
\end{exa}

The next example seems interesting \textit{per se}. Not only it gives a family
of linearly ordered semigroups which are neither
abelian nor groups (at least in general); it also shows that, for
each
$n$, certain subsemigroups of ${\rm GL}_n(\mathbb R)$ consisting of triangular matrices are linearly orderable.

\begin{exa}\label{exa:matrices}
We let a semiring be a
triple $(A,+,\cdot)$ consisting of a set $A$ and
associative operations $+$ and $\cdot$ from $A \times A$ to $A$ (referred to, respectively, as the semiring addition and multiplication) such that
\begin{enumerate}[label={\rm\arabic{*}.}]
\item $(A,+)$ is an abelian monoid, whose identity we denote by $0$;
\item $0$ annihilates $A$, that is $0 \cdot a = a \cdot 0 =
0$ for every $a \in A$;
\item multiplication distributes over addition, that is $a(b+c) = ab + ac$ and $(a + b) c = ac + b c$ for all $a,b,c \in A$.
\end{enumerate}
(In other words, a sem\-i\-ring is just a
ring where elements
do not need have an additive inverse.) We call $(A, +)$ and $(A, \cdot)$, respectively, the additive monoid and the
multiplicative semigroup of $(A,+,\cdot)$, which in turn is
termed a unital semiring if $(A, \cdot)$ is a monoid too; see \cite[Ch.
II]{Hebisch98} and \cite[Ch. 1, p. 1]{Golan99}.

A semiring $(A, +, \cdot)$ is said to be orderable
if there exists a (total) order $\le$ on $A$
such that $(A, +, \le)$ and $(A, \cdot, \le)$ are ordered semigroups, in which case $(A, +, \cdot, \le)$ is referred to as an ordered semiring. If, on the other hand, the following hold:
\begin{enumerate}
\item[4.] $(A, +, \le)$ is a linearly ordered monoid;
\item[5.] $a c < b c$ and $c a < c b$ for all $a,b,c \in A$ with $a < b$ and $0 < c$,
\end{enumerate}
then $(A, +, \cdot)$ is said to be linearly orderable and $(A, +, \cdot, \le)$
is called a linearly ordered semiring; cf. \cite[Ch. 20]{Golan99}. Common examples of linearly ordered semirings
are the [non-negative] integers, the [non-negative] rational numbers, and the [non-negative] reals with their usual addition, multiplication, and order.

With that said, let $\mathbb A = (A, +, \cdot)$ be a fixed semiring. We write $\mathcal M_n(A)$ for the set of
$n$-by-$n$ matrices with entries in $A$. Endowed with the usual operations of
entry-wise addition and row-by-column multiplication induced by the
structure of $\AA$, here respectively denoted by the same
symbols as the addition and multiplication of the latter, $\mathcal M_n(A)$
becomes itself a semiring, which we call the semiring of $n$-by-$n$ matrices over $\AA$ and write as $\mathcal M_n(\mathbb A)$; see \cite[Ch. 3]{Golan99}.

Suppose now that $\mathbb A$ is linearly ordered by a certain order $\le$, in such a way that $\AA_\sharp = (A, +, \cdot, \le)$ is a linearly ordered semiring, and
denote by ${\rm U}_n(\AA_\sharp^+)$ the subsemigroup of the multiplicative
semigroup of $\mathcal M_n(\mathbb A)$ consisting of
all upper triangular matrices whose entries on or above the main diagonal belong to
$$
\AA_\sharp^+ = \{a \in A: 0 < a\}.
$$
We observe that ${\rm U}_n(\AA_\sharp^+)$ is not a group (and
not even a monoid) for $n \ge 2$. But what is perhaps more interesting is the following:
\begin{thm}\label{th:zig_zag_ordering}
${\rm U}_n(\AA_\sharp^+)$ is a linearly orderable semigroup.
\end{thm}
\begin{proof}
Set $I_n = \{1, 2, \ldots, n\}$, $\Xi_n = \{(i,j) \in
I_n \times I_n: i \le j\}$ and define a binary relation $\le_n$ on
$\Xi_n$ by $(i_1, j_1) \le_n
(i_2, j_2)$ if and only if (i) $j_1 - i_1
< j_2 - i_2$ or (ii) $j_1 - i_1 = j_2 - i_2$ and $j_1 < j_2$. It is seen
that $\le_n$ is a well-order, so we can
define a binary relation $\le_{n,\rm U}$ on ${\rm U}_n(
\AA_\sharp^+)$ by taking, for $\alpha = (a_{i,j})_{i,j=1}^n$ and $\beta =
(b_{i,j})_{i,j=1}^n$ in ${\rm U}_n(\AA_\sharp^+)$, $\alpha
\le_{n,\rm U} \beta$ if and only if (i) $\alpha = \beta$ or (ii) there exists
$(i_0,j_0) \in \Xi_n$ such that $a_{i_0,j_0} <
b_{i_0,j_0}$ and $a_{i,j} = b_{i,j}$ for all $(i,j) \in \Xi_n$ with
$(i,j) <_n (i_0, j_0)$.

It is straightforward that $\le_{n,\rm U}$ is an order. To see, in particular, that it is total: Pick $\alpha = (a_{i,j})_{i,j=1}^n$ and $\beta =
(b_{i,j})_{i,j=1}^n$ in ${\rm U}_n(\AA_\sharp^+)$ with $\alpha \ne \beta$. There then exists $(i_0, j_0) \in \Xi_n$ such that $a_{i_0,j_0} \ne b_{i_0,j_0}$, where $(i_0, j_0)$ is chosen in such a way that $a_{i,j} = b_{i,j}$ for every $(i,j) \le_n (i_0, j_0)$. Since $\le$ is total, we have that either $\alpha <_{n,\rm U} \beta$ if $a_{i_0,j_0} < b_{i_0,j_0}$ or $\beta <_{n,\rm U} \alpha$ otherwise, and we are done.

It remains to prove that ${\rm U}_n(\AA_\sharp^+)$ is linearly ordered by $\le_{n,\rm U}$. For, let $\alpha$ and $\beta$ be as above and suppose $\alpha
<_{n,\rm U} \beta$, viz there exists $(i_0, j_0) \in \Xi_n$ with
$a_{i_0,j_0} < b_{i_0,j_0}$ and $a_{i,j} = b_{i,j}$ for all $(i,j) \in
\Xi_n$ with $(i,j) <_n (i_0, j_0)$. Given $\gamma = (c_{i,j})_{i,j=1}^n$ in ${\rm U}_n(\AA_\sharp^+)$ we then have $a_{i,k} c_{k,j}
\le b_{i,k} c_{k,j}$ and $ c_{i,k} a_{k,j} \le c_{i,k} b_{k,j}$ for all
$(i,j) \in \Xi_n$ and $k \in I_n$ such that $(i,k) \leq_n (i_0, j_0)$ and $(k,j)
\le_n (k,j_0)$, and in fact $a_{i_0,j_0} c_{j_0,j_0} < b_{i_0,j_0}
c_{j_0,j_0}$
and $ c_{i_0,i_0} a_{i_0,j_0} < c_{i_0,i_0} b_{i_0,j_0}$ since $(A, +, \cdot, \le)$ is a linearly ordered semiring. It follows that, for all $(i,j)
\in \Xi_n$ with $(i,j) \le_n (i_0, j_0)$,
\begin{displaymath}
\textstyle \sum_{k=1}^n a_{i,k} c_{k,j} = \sum_{k=i}^j
a_{i,k} c_{k,j} \le \sum_{k=i}^j b_{i,k} c_{k,j} = \sum_{k=1}^n
b_{i,k} c_{k,j}
\end{displaymath}
and, similarly, $\sum_{k=1}^n c_{i,k} a_{k,j} \le \sum_{k=1}^n
c_{i,k} b_{k,j}$. In particular, these majorations are equalities for
$(i,j) <_n (i_0, j_0)$ and strict inequalities if $(i,j) = (i_0, j_0)$. So $\alpha \gamma <_{n,\rm U} \beta
\gamma$ and $ \gamma \alpha <_{n,\rm U} \gamma  \beta$, and the proof is complete.
\end{proof}
We refer to the order $\le_{n,\rm U}$ defined in the proof of Theorem
\ref{th:zig_zag_ordering} as the \textit{zig-zag order} on ${\rm U}_n(\AA_\sharp^+)$.
If ${\rm L}_n(\AA_\sharp^+)$ is the
subsemigroup of the multiplicative semigroup of $\mathcal M_n(\mathbb A)$
consisting of all \textit{lower} triangular matrices whose entries on or below the main diagonal are in $\AA_\sharp^+$, it is then easy to see that ${\rm L}_n(\AA_\sharp^+)$ is
itself linearly orderable: It is, in fact, linearly ordered by the binary
relation $\le_{n,\rm L}$ defined by taking $\alpha \le_{n,\rm L} \beta$ if and only if $\alpha^\top
\le_{n,\rm U} \beta^\top$, where the superscript `$\top$' stands for `transpose'.
If ${\rm T}_n(\AA_\sharp^+)$ is the
subsemigroup of $(\mathcal M_n(A), \cdot)$
generated by ${\rm U}_n(\AA_\sharp^+)$ and ${\rm L}_n(\AA_\sharp^+)$, it
is hence natural to ask the following:
\begin{question}
Is ${\rm T}_n(\AA_\sharp^+)$ a linearly orderable semigroup?
\end{question}
While at present we do not have an answer to this, it was remarked by
Carlo Pagano (Universit\`a di Roma Tor Vergata, Italy) in a private
communication that $\mathcal M_n(\AA_\sharp^+)$, namely the subsemigroup of
$(\mathcal M_n(A), \cdot)$ consisting of \textit{all} matrices with entries in
$\AA_\sharp^+$, is not in general linearly orderable:
For a specific counterexample,
let $\AA_\sharp$ be the real field together with its usual order, and take as $\alpha$ the $n$-by-$n$ real matrix
whose entries are all equal to $1$ and as $\beta$ \textit{any} $n$-by-$n$ matrix with
positive real entries each of whose columns has sum equal to $n$. Then $\alpha^2 = \alpha
\beta$.

Apparently, the question has not been addressed
before by other authors,
although the ordering of $\mathcal
M_n(\mathbb A)$, in the case where $\mathbb A$ is a \textit{partially} orderable
semiring, is considered in \cite[Example 20.60]{Golan99}.
\end{exa}
\begin{exa}
In what follows, we let $\mathbb K = (K, +, \cdot)$ be a semiring (see Example \ref{exa:matrices} for the terminology) and $\mathbb A = (A, \diamond)$ a semigroup, and use $K[A]$ for the set of all functions $f: A \to K$ such that $f$ is finitely supported in $\mathbb K$, namely $f^{-1}(0_K)$ is a finite subset of $A$, where $0_K$ is the additive identity of $\mathbb K$.

In fact, $K[A]$ can be turned into a semiring, here written as $\mathbb K[\mathbb A]$, by endowing it with the operations of pointwise addition and Cauchy product induced by the structure of $\mathbb A$ and $\mathbb K$ (these operations are denoted below with the same symbols as the addition and the multiplication of $\mathbb K$, respectively). We have:
\begin{thm}\label{thm:semigroup_semiring}
Suppose $\mathbb K$ is a linearly orderable semiring and $\mathbb A$ is a linearly orderable semigroup. Then $\mathbb K[\mathbb A]$ is a linearly orderable semiring too.
\end{thm}
\begin{proof}
The claim is obvious if $A = \emptyset$, so assume that $A$ is non-empty, and let $\le_K$ and $\le_A$ be, respectively, orders on $A$ and $K$ for which $(K, +, \cdot, \le_K)$ is a linearly ordered semiring and $(A, \diamond, \le_A)$ a linearly ordered semigroup.

Then, given $\alpha \in A$ and $f \in K[A]$, we let $f_{\downarrow \alpha}$ (respectively, $f_{\uparrow \alpha}$) be the function $A \to K$ taking $a$ to $f(a)$ if $a <_A \alpha$ (respectively, $\alpha \le_A a$), and to $0_K$ otherwise, in such a way that $f = f_{\downarrow \alpha} + f_{\uparrow \alpha}$. Also, we denote by $\mu$ the map $K[A] \times K[A] \to A \cup \{A\}$ sending a pair $(f,g)$ to $\min\{a \in A: f(a) \ne g(a)\}$ if $f \ne g$ (the minimum is taken with respect to  $\le_A$, and it exists by consequence of the definition itself of $K[A]$), and to $A$ otherwise.

We define a binary relation $\le$ on $K[A]$ by letting $f \le g$ if and only if either $f = g$ or $f \ne g$ and $f(\mu(f,g)) <_K f(\mu(f,g))$. It is clear that $\le$ is a total order on $K[A]$, and we want to prove that it is also compatible with the algebraic structure of $\mathbb K[\mathbb A]$, in the sense that $\mathbb K[\mathbb A]$ is linearly ordered by $\le$.

For, pick $f,g,h \in K[A]$ with $f < g$.
Since the additive monoid of $\mathbb K$ is linearly ordered by $\le_K$, we have $\mu(f, g) = \mu(f +
h, g + h)$, and thus $f + h < g + h$. That is, $(K[A], +, \le)$ is a linearly ordered monoid in its own right. On another hand, assume $\Theta < h$, where $\Theta$ is the function $A \to K: a \mapsto 0_K$, and set $\alpha = \mu(f,g)$ and $\beta = \mu(\Theta, h)$.
We have $ f_{\downarrow \alpha} = g_{\downarrow \alpha}$ and $h = h_{\uparrow \beta}$, with the result that $f h < g h$ if and only if $f_{\uparrow \alpha} h_{\uparrow \beta} < g_{\uparrow \alpha} h_{\uparrow \beta}$, and the latter inequality is certainly true, since on the one side $ f_{\uparrow \alpha} h_{\uparrow \beta}(a) = g_{\uparrow \alpha}  h_{\uparrow \beta}(a) = 0_K$ for $a <_A \alpha \diamond \beta$, and on the other side
\begin{displaymath}
f_{\uparrow \alpha} h_{\uparrow \beta}(\alpha \diamond \beta) = f_{\uparrow \alpha}(\alpha)  h_{\uparrow \beta}(\beta) <_K g_{\uparrow \alpha}(\alpha) h_{\uparrow \beta}(\beta) = g_{\uparrow \alpha} h_{\uparrow \beta}(\alpha \diamond \beta).
\end{displaymath}
In a similar way, it is seen that $h f < h g$. So, by the arbitrariness of $f$, $g$, and $h$, we get that $(K[A], +, \cdot, \le)$ is a linearly ordered semiring.
\end{proof}
So taking $\mathbb A$ to be the free commutative monoid (respectively, the free monoid) on a certain set and recalling that free groups (and hence free monoids) are linearly orderable (Example \ref{exa:torsion-freeness}), we have:
\begin{cor}
The semiring $\mathbb K$ is linearly orderable if and only if the same is true for the semiring of polynomials over $\mathbb K$ depending on a given set of pairwise commuting (respectively, non-commuting) variables.
\end{cor}
\end{exa}

\section*{Acknowledgments}
The author is indebted with Martino Garonzi (Universit\`a di Padova, Italy) for
having attracted his attention to the work of G. A. Freiman, M. Herzog and
coauthors by which this research was inspired. Also, he is grateful to Alain Plagne (CMLS, \'Ecole
polytechnique, France) for uncountably many
suggestions and to Carlo Sanna (Universit\`a di Torino, Italy) for an accurate check of the proof of Theorem \ref{th:main}. Last but not least, he would like to thank the anonymous referees for valuable comments which improved the quality of the paper (most notably, this is the case with Theorem \ref{thm:semigroup_semiring}, initially stated and proved by the author in a less general form).

\end{document}